\documentclass[preprint,12pt,author]{elsarticle}

\usepackage{amssymb}
\usepackage[numbers]{natbib}
\usepackage[usenames]{color}

\usepackage[colorlinks=true,
linkcolor=webgreen,
filecolor=webbrown,
citecolor=webgreen]{hyperref}

\definecolor{webgreen}{rgb}{0,.5,0}
\definecolor{webbrown}{rgb}{.6,0,0}

\usepackage{fullpage}
\usepackage{float}

\usepackage{graphics,amsmath,amssymb}
\usepackage{graphicx}
\usepackage{amsthm}
\usepackage{amsfonts}
\usepackage{latexsym}
\usepackage{epsf}
\usepackage[dvipsnames]{xcolor}

\DeclareMathOperator{\GL}{GL}
\DeclareMathOperator{\SL}{SL}

\DeclareMathOperator{\R}{\mathbb{R}}
\DeclareMathOperator{\C}{\mathbb{C}}
\DeclareMathOperator{\Z}{\mathbb{Z}}

\DeclareMathOperator{\re}{Re}
\DeclareMathOperator{\im}{Im}

\DeclareMathOperator{\PGL}{PGL}

\usepackage{lineno,hyperref}
\modulolinenumbers[5]

\journal{Journal of Number Theory}

\begin{document}

\theoremstyle{plain}
\newtheorem{theorem}{Theorem}
\newtheorem{corollary}[theorem]{Corollary}
\newtheorem{lemma}[theorem]{Lemma}
\newtheorem{proposition}[theorem]{Proposition}

\theoremstyle{definition}
\newtheorem{definition}[theorem]{Definition}
\newtheorem{example}[theorem]{Example}
\newtheorem{conjecture}[theorem]{Conjecture}

\theoremstyle{remark}
\newtheorem{remark}[theorem]{Remark}

\begin{frontmatter}

\title{Subconvexity for $\GL_{3}(\R)$ $L$-Functions:\\
The Key Identity via Integral Representations}
\author{Raphael Schumacher}
\address{Department of Mathematics, ETH Zurich, R\"amistrasse 101, 
8092 Zurich, Switzerland}
\ead[E-mail address]{raphael.schumacher@math.ethz.ch}

\begin{abstract}
We study the subconvexity problem for $\GL_{3}(\R)$ $L$-functions in the t-aspect using integral representations by 
combining techniques employed by Michel--Venkatesh in their study of the corresponding problem for $\GL_{2}$ with ideas from recent works of Munshi, Holowinsky--Nelson and Lin. Our main objective is to give -- from the perspective of integral representations of $L$-functions and automorphic representation theory -- a possible explanation of the origin of the ``key identity'' arising in the latter series of works.
\end{abstract}


\end{frontmatter}

\section{Introduction}
\label{Introduction}

The subconvexity problem for standard automorphic $L$-functions on $\GL_{2}$ 
over arbitrary number fields was completely solved in the paper of Philippe Michel and Akshay Venkatesh \cite{14}. They used suitably truncated integral representations of the corresponding $L$-functions \cite[\text{Lemma 5.1.4}]{14}, \cite[\text{Lemma 11.9}]{20} and dynamical arguments 
specific to $\GL_{2}$ to obtain their main theorem. The starting point of their argument, specialized to the case of automorphic forms $\varphi$ on $\GL_{2}(\Z)\backslash\GL_{2}(\R)$, is the global zeta integral $Z(\varphi,\tfrac{1}{2}+iT)$ given by
\begin{displaymath}
\begin{split}
Z(\varphi,\tfrac{1}{2}+iT)&=\int_{y\in\mathbb{R}^{\times}/\mathbb{Z}^{\times}
}\varphi(a(y))y^{iT}d^{\times}y\;\;\;\;\text{with $a(y):=\left(\begin{matrix}y&0\\0&1\end{matrix}\right)$}
\end{split}
\end{displaymath}
and a basic unfolding principle relating global and local zeta functionals with their associated $L$-functions. The main step in their proof is the estimation of the global functional from above for a suitable choice of vector $\varphi$ for which the local functional has a good lower bound. Their method for the study of the subconvexity problem on $\GL_{2}(\R)$ is therefore based on integral representations of $L$-functions and -- to our knowledge -- has never been used for $\GL_{3}(\R)$. Our Theorem \ref{Bounds for Z(varphi,1/2+iT) and Z(W_varphi,1/2+iT) on GL_3(R)} below and its proof, given later in this paper, is an adaptation of parts of their ideas to $\GL_{3}(\R)$, since it contains an analogous construction of $\varphi$ and similar upper and lower bounds for the corresponding global and local zeta integrals appearing in the $\GL_{3}(\R)$-problem.\\

The first subconvex bounds for $\GL(1)$ twists of a fixed Hecke--Maass cusp form on $\GL(3)$, not necessarily self-dual, were obtained by Munshi in \cite{16,17}, who used a completely different strategy. His technique was subsequently simplified by Holowinsky--Nelson \cite{5} for the $q$-aspect and by Lin \cite{13} for the $t$-aspect. The authors of these papers discovered, through a careful study of Munshi's work (see \cite[App.\;B]{5}), a ``key identity" implicit in his papers underlying the success of his method. By extracting this key identity, which amounts to the Poisson summation formula applied to an incomplete Gauss sum, they were able to streamline the method and improve the exponent. For the $t$-aspect as addressed by Lin \cite{13}, the relevant key identity is
\begin{equation}\label{key identity 1}
\begin{split}
n^{-iT}V_{A}\left(\frac{2\pi n}{N}\right)
&=\left(\frac{2\pi}{T}\right)^{iT}\left(\frac{\ell}{p}\right)^{1-iT}e\left(\frac{T}{2\pi}\right)\frac{T^{3/2}}{N}\sum_{r=1}^{\infty}r^{-iT}e\left(-\frac{np}{\ell r}\right)V\left(\frac{r}{Np/\ell T}\right)+O\left(T^{1/2-A}\right)
\\
&\quad-\left(\frac{2\pi}{N}\right)^{iT}e\left(\frac{T}{2\pi}\right)T^{1/2}\sum_{\substack{r\in\Z\\r\neq0}}\int_{\R_{+}^{\times}}x^{-iT}e\left(-\frac{nT}{Nx}\right)V(x)e\left(-\frac{rNp}{\ell T}x\right)dx,
\end{split}
\end{equation}
where $V$ and $V_{A}$ are fixed smooth cutoffs (see \cite{13} for details). As explained in \cite{5,13}, the key identity leads quickly to a subconvex bound after an amplification step and some fairly standard manipulations. The authors of \cite{5,13} could not explain the origins of their key identities and left open the question of whether there might be a natural way to discover their usefulness.\\

In this paper, we give a possible explanation of the origin of Lin's key identity \eqref{key identity 1} and its application to subconvexity from the perspective of integral representations of $L$-functions attached to automorphic representations $\pi$ on $\GL_{3}(\R)$. By combining the different approaches of Michel--Venkatesh and of Munshi--Holowinsky--Nelson--Lin, we obtain an integral-representation-theoretic derivation of the $t$-aspect key identity \eqref{key identity 1}. This allows us to re-prove Lin's subconvex bound by using his results \cite{13} and to put his strategy on an automorphic foundation. The deduction and explanation of Lin's key identity with representation-theoretic methods is the main result of this paper together with the following theorem:
\begin{theorem}\label{Bounds for Z(varphi,1/2+iT) and Z(W_varphi,1/2+iT) on GL_3(R)}(Bounds for the global and local zeta integrals on $\GL_{3}(\R)$)\\
Let $\varphi\in\pi$ be the automorphic function corresponding to the Whittaker function $W_{\varphi}\in\mathcal{W}(\pi)$ constructed in \S\ref{suitable Whittaker function} and let $\varphi^{1}$ be the first projection of $\varphi$ defined in \eqref{first projection formula}. Then the global zeta integral $Z(\varphi,\tfrac{1}{2}+iT)$ and the local zeta integral $Z(W_{\varphi},\tfrac{1}{2}+iT)$ for $\GL_{3}(\R)$ satisfy the following bounds:
\begin{displaymath}
\begin{split}
Z(\varphi,\tfrac{1}{2}+iT)=\int_{\mathbb{R}^{\times}/\mathbb{Z}^{\times}}\varphi^{1}\left[\left(\begin{matrix}y&0&0\\0&1&0\\0&0&1\end{matrix}\right)\right]y^{iT-1/2}d^{\times}y&\ll T^{1/4-1/36+\varepsilon}
\end{split}
\end{displaymath}
and
\begin{displaymath}
\begin{split}
Z(W_{\varphi},\tfrac{1}{2}+iT)=\int_{\mathbb{R}^{\times}}W_{\varphi}\left[\left(\begin{matrix}y&0&0\\0&1&0\\0&0&1\end{matrix}\right)\right]|y|^{iT-1/2}d^{\times}y\asymp T^{-1/2}\;\;\;\;\text{as $T\rightarrow\infty$}.\end{split}
\end{displaymath}
\end{theorem}
\noindent We will use Lin's results and his key identity \eqref{key identity 1} to prove the first estimate of Theorem \ref{Bounds for Z(varphi,1/2+iT) and Z(W_varphi,1/2+iT) on GL_3(R)}; the second estimate combined with the first directly implies Lin's bound \cite{13}
\begin{displaymath}
\begin{split}
L(\pi,\tfrac{1}{2}+iT)&=\frac{Z(\varphi,\tfrac{1}{2}+iT)}{Z(W_{\varphi},\tfrac{1}{2}+iT)}\ll T^{3/4-1/36+\varepsilon}\;\;\;\;\text{as $T\rightarrow\infty$}.\end{split}
\end{displaymath}
The proof of this bound, given later in this paper, shows where Lin's key identity comes from. Our explanation of its origin is new and is not recognizable in Lin's approach.\\

It remains to outline how we explain the origin of the key identity: The starting point of our explication is the Whittaker function $W_{\varphi}$ defined in a simple way in terms of an additive character and a compactly supported bump function such that $W_{\varphi}$ satisfies the key property
\begin{displaymath}
\begin{split}
W_{\varphi}\left[\left(\begin{matrix}y&0&0\\0&1&0\\0&0&1\end{matrix}\right)\right]&=T^{3/4}e\left(-\frac{y}{\sqrt{T}}\right)V_{0}\left(\frac{y}{T^{3/2}}\right).
\end{split}
\end{displaymath}
The motivation for this choice is a stationary phase analysis showing -- by the interaction between the additive and the multiplicative character -- that $Z(W_{\varphi},\frac{1}{2}+iT)\asymp T^{-1/2}$. Our choice of $W_{\varphi}$ is exactly analogous to the $\GL_{2}(\R)$-case in \cite[\text{Lemma 5.1.4}]{14} and might thus be regarded as natural. After the construction of the vector $\varphi$ coming from the Whittaker function $W_{\varphi}$, by the use of a Fourier-Whittaker expansion, we truncate the global zeta integral 
$Z(\varphi,\frac{1}{2}+iT)$ on $\GL_{3}(\R)$ to a suitable bounded interval $I\subset\R_{+}^{\times}$. Next, we unfold the truncated global zeta integral to obtain a bound roughly of the form
\begin{displaymath}
\begin{split}
Z(\varphi,\tfrac{1}{2}+iT)&\lessapprox\sum_{n=1}^{\infty}\frac{a(1,n)}{n}w_{0}\left(\frac{n}{T^{3/2}}\right)\int_{y\in\R_{+}^{\times}}W_{\varphi}\left[\left(\begin{matrix}ny&0&0\\0&1&0\\0&0&1\end{matrix}\right)\right]y^{iT-1/2}d^{\times}y,
\end{split}
\end{displaymath}
where the $a(1,n)$'s are the Fourier-Whittaker coefficients of $\varphi$ and $w_{0}\in C_{c}^{\infty}([1,2])$.\\
After the change of variables $y:=\frac{1}{x}$ in the above integral and by the use of the above key property satisfied by our Whittaker function $W_{\varphi}$, we obtain
\begin{displaymath}
\begin{split}
\int_{\R_{+}^{\times}}W_{\varphi}\left[\left(\begin{matrix}\frac{n}{x}&0&0\\0&1&0\\0&0&1\end{matrix}\right)\right]x^{1/2-iT}d^{\times}x&=T^{3/4}\int_{\R_{+}^{\times}}e\left(-\frac{n}{\sqrt{T}x}\right)V_{0}\left(\frac{n}{T^{3/2}x}\right)x^{1/2-iT}d^{\times}x.
\end{split}
\end{displaymath}
The second integral in the above equation can now be recognized as the main integral
\begin{displaymath}
\begin{split}
M=\int_{\R_{+}^{\times}}x^{-iT}e\left(-\frac{nT}{Nx}\right)V(x)dx
\end{split}
\end{displaymath}
in the proof of Lin's key identity \eqref{key identity 1}, if we set $n\asymp N:=T^{3/2}$ and $V(x):=\frac{1}{\sqrt{x}}V_{0}\left(\frac{n}{T^{3/2}x}\right)$. Exactly this observation justifies the above change of variables $y:=\frac{1}{x}$ and allows us further computations.\\
At this point, we replace the integral $M$ by the corresponding Riemann sum
\begin{displaymath}
\begin{split}
R&=h\sum_{r=1}^{\infty}(hr)^{-iT}e\left(-\frac{nT}{Nhr}\right)V(hr),\;\;\text{where $h>0$}
\end{split}
\end{displaymath}
plus its oscillations (error terms) around $M$'s exact value
\begin{displaymath}
\begin{split}
O=-\sum_{\substack{r\in\Z\\r\neq0}}\int_{\R_{+}^{\times}}x^{-iT}e\left(-\frac{nT}{Nx}\right)V(x)e\left(-\frac{rx}{h}\right)dx.
\end{split}
\end{displaymath}
These error terms are obtained with the Poisson summation formula applied to the Riemann sum $R$ as in \cite{13}. Due to the above discretization of the integral $M$, we have with the choice of $h=\frac{\ell T}{Np}$ the following new manifestation of the key identity
\begin{equation}\label{key identity 2}
\begin{split}
\int_{\R_{+}^{\times}}x^{-iT}e\left(-\frac{nT}{Nx}\right)V(x)dx&=\left(\frac{\ell T}{Np}\right)^{1-iT}\sum_{r=1}^{\infty}r^{-iT}e\left(-\frac{np}{\ell r}\right)V\left(\frac{r}{Np/\ell T}\right)\\
&\quad-\sum_{\substack{r\in\Z\\r\neq0}}\int_{\R_{+}^{\times}}x^{-iT}e\left(-\frac{nT}{Nx}\right)V(x)e\left(-\frac{rNp}{\ell T}x\right)dx,
\end{split}
\end{equation}
which is equivalent to Lin's \eqref{key identity 1}, if we apply a stationary phase analysis to the integral $M$ in the formula \eqref{key identity 2}. This quantization shows, that it is not just the phase $n^{-iT}$ for which the key identity is a substitute as in \cite{13}, but that it is in fact the whole integral $M$ which the key identity replaces. For this reason, our Whittaker function $W_{\varphi}$ relying on basic automorphic principles is able to determine the key identity relevant for the $t$-aspect subconvexity problem on $\GL_{3}$. This is the main observation and led us to the above explanation of the origin of the key identity in a natural and simple way.\\

We hope that the structural perspective suggested here may be useful in identifying analogous phenomena for $\GL_{n}(\mathbb{R})$, in uncovering a larger connection between subconvexity and integral representations of $L$-functions, and perhaps in a better understanding of the so called $\delta$-methods \cite{1,13,16,17} of which the key identity may be understood as a special case. Our hope is also that the insights gained in this paper will be useful for extending the method of Munshi and its simplifications by Holowinsky--Nelson and Lin to more general pairs $(\GL(n),\GL(m))$ and to get a deeper understanding of the subconvexity problem in general. We expect that our method will also work for the $q$-aspect, where it should reproduce the key identity of \cite{5}, and that the method will apply in the adelic setting as well, leading to a bound, which is uniform in all aspects of the twisting character $\chi$. This is work in progress.\\

This paper is organized as follows: Section \ref{Definitions and Basic Facts} introduces important definitions and basic facts, which will be needed through the paper. In section \ref{The Geometric Approximate Functional Equation for GL_{3}(R)} we present and prove the new truncated integral representation for the global zeta integral for $\GL_{3}(\R)$, which relies on a stationary phase computation for the special choice of the Whittaker function $W_{\varphi}$ presented in the subsections \ref{suitable Whittaker function} and \ref{stationary phase}. This new identity reduces the subconvexity problem for $\GL_{3}$ to an estimation of a period integral of the vector $\varphi$ and is the heart of section \ref{The Geometric Approximate Functional Equation for GL_{3}(R)}. In section \ref{Subconvexity for GL_{3}(R) L-Functions via Integral Representations} we explain the possible origin of our key identity \eqref{key identity 2} and finally, we use Lin's results \cite{13} to conclude.

\section{Definitions and Basic Facts}
\label{Definitions and Basic Facts}

As usual we denote $\exp(2\pi ix)$ by $e(x)$. We will use the variable $\varepsilon>0$ to denote an arbitrarily small positive constant, which may change from line to line. The notation $A\ll B$ will mean that $|A|\leq C|B|$ for some constant $C$. The notation $A\asymp B$ will mean that $B/T^{\varepsilon}\ll A\ll BT^{\varepsilon}$. We will also use the space $\R_{+}^{\times}:=(0,\infty)\cong\R^{\times}/\Z^{\times}$ with the corresponding measure $d^{\times}y:=\frac{dy}{|y|}$.\\
Let $\pi$ be an automorphic cuspidal representation of $\GL_{3}(\R)$ and let $\mathcal{W}(\pi)$ and $\mathcal{K}(\pi)$ be its unique Whittaker and Kirillov models with respect to the additive character
\begin{displaymath}
\begin{split}
\psi:\left(\begin{matrix}1&x&z\\0&1&y\\0&0&1\end{matrix}\right)\mapsto e(x+y).
\end{split}
\end{displaymath}
Let $\varphi\in\pi\subset L^{2}\left(\GL_{3}(\Z)\backslash\GL_{3}(\R)\right)$ be a fixed Hecke-Maass cusp form on $\GL_{3}(\mathbb{Z})$.\\
We define the matrix element $a(y)$ by
\begin{displaymath}
\begin{split}
a(y):&=\left(\begin{matrix}y&0&0\\0&1&0\\0&0&1\end{matrix}\right)\in\GL_{3}(\R).
\end{split}
\end{displaymath}
The following definitions and theorems are sometimes modified versions of the definitions and theorems given in the corresponding references. The little modifications are necessary to make the following computations in this article as natural as possible.

\begin{definition}(The Whittaker function for $\GL_{3}(\R)$)\;\cite[\text{pp.\;180--181}]{7}, \cite[\text{pp.\;235--236}]{8}\\
Let $\varphi\in\pi$. We define the Whittaker function $W_{\varphi}\in\mathcal{W}(\pi)$ corresponding to $\varphi\in\pi$ by
\begin{displaymath}
\begin{split}
W_{\varphi}(g):&=\int_{0}^{1}\int_{0}^{1}\int_{0}^{1}\varphi\left[\left(\begin{matrix}1&x&z\\0&1&y\\0&0&1\end{matrix}\right)g\right]e(-x-y)dxdydz.
\end{split}
\end{displaymath}
\end{definition}

\begin{definition}\label{first projection}(The first projection $\varphi^{1}$ of $\varphi$)\;\cite[\text{(8.8)}]{2}, \cite{3}, \cite{9}, \cite[\text{(7.2)\;and\;(7.7)}]{15}\\
We define the first projection $\varphi^{1}$ of the automorphic form $\varphi$ by
\begin{equation}\label{first projection formula}
\begin{split}
\varphi^{1}(g):&=\int_{u\in\R/\Z}\int_{v\in\R/\Z}\varphi\left[\left(\begin{matrix}1&0&u\\0&1&v\\0&0&1\end{matrix}\right)g\right]e(-v)dudv.
\end{split}
\end{equation}
We have the Fourier-Whittaker expansion \cite[\text{(7.7)}]{15}
\begin{displaymath}
\begin{split}
\varphi^{1}(g)&=\sum_{\substack{n=-\infty\\n\neq0}}^{\infty}\frac{a(1,|n|)}{|n|}W_{\varphi}\left[\left(\begin{matrix}n&0&0\\0&1&0\\0&0&1\end{matrix}\right)g\right],
\end{split}
\end{displaymath}
with Fourier-Whittaker coefficients $a(1,|n|)\in\C$.\\
In particular, by setting $g:=a(y)$, we obtain the expansion \cite[\text{(8.8)}]{2}
\begin{equation}\label{Fourier-Whittaker expansion}
\begin{split}
\varphi^{1}(a(y))&=\sum_{\substack{n=-\infty\\n\neq0}}^{\infty}\frac{a(1,|n|)}{|n|}W_{\varphi}(a(ny)).
\end{split}
\end{equation}
\end{definition}

\begin{definition}(The dual automorphic form $\widetilde{\varphi}$)\;\cite[\text{pp.\;180--181}]{7}, \cite[\text{pp.\;235--236}]{8}\\
We define the automorphic form $\widetilde{\varphi}$, which is dual to the automorphic form $\varphi$ by the expression
\begin{displaymath}
\begin{split}
\widetilde{\varphi}(g):&=\varphi\left({}^{t}g^{-1}\right)=\varphi\left(w\cdot{}^{t}g^{-1}\right),
\end{split}
\end{displaymath}
where
\begin{displaymath}
\begin{split}
w:&=\left(\begin{matrix}0&0&1\\0&-1&0\\1&0&0\end{matrix}\right)\in\SL_{3}(\R)\subset\GL_{3}(\R).
\end{split}
\end{displaymath}
\end{definition}

\begin{definition}\label{dual Whittaker function}(The dual Whittaker function $\widetilde{W}_{\varphi}$ belonging to $\widetilde{\varphi}$)\;\cite[\text{p.\;181}]{7}, \cite[\text{pp.\;235--236}]{8}\\
We define the dual Whittaker function $\widetilde{W}_{\varphi}$ corresponding to the automorphic form $\widetilde{\varphi}$ by
\begin{displaymath}
\begin{split}
\widetilde{W}_{\varphi}(g):&=W_{\varphi}(w\cdot{}^{t}g^{-1})=W_{\widetilde{\varphi}}(g),
\end{split}
\end{displaymath}
where the matrix element $w$ is as above.\\
This means that we have the Fourier-Whittaker expansion \cite[\text{(13.1.6)}]{8}
\begin{equation}\label{Dual form Whittaker expansion}
\begin{split}
\widetilde{\varphi}^{1}(g)&=\sum_{\substack{n=-\infty\\n\neq0}}^{\infty}\frac{a(|n|,1)}{|n|}\widetilde{W_{\varphi}}\left[\left(\begin{matrix}n&0&0\\0&1&0\\0&0&1\end{matrix}\right)g\right].
\end{split}
\end{equation}
\end{definition}

\begin{theorem}\label{projection identity}(The $\GL_{3}(\R)$ projection identity)\;\cite[\text{(13.3.1)}]{8}\\
Let $\varphi\in\pi$ be an automorphic form on $\GL_{3}(\R)$. We have
\begin{displaymath}
\begin{split}
\varphi^{1}(g)&=\int_{x\in\R}\widetilde{\varphi}^{1}\left[\left(\begin{matrix}1&0&0\\x&1&0\\0&0&1\end{matrix}\right)\cdot w'\cdot{}^{t}g^{-1}\right]dx,
\end{split}
\end{displaymath}
where $w'$ is given by
\begin{displaymath}
\begin{split}
w':&=\left(\begin{matrix}-1&0&0\\0&0&-1\\0&1&0\end{matrix}\right)\in\GL_{3}(\R).
\end{split}
\end{displaymath}
\end{theorem}

\begin{definition}\label{The global Zeta integrals}(The global zeta integrals for $\GL_{3}(\R)$)\;\cite[\text{(0.1.4)}]{7}, \cite[\text{p.\;240}]{8}, \cite{15}\\
We define the two global zeta integrals $Z(\varphi,s)$ and $\widetilde{Z}(\widetilde{\varphi},s)$ for $s\in\mathbb{C}$ by
\begin{displaymath}
\begin{split}
Z(\varphi,s):&=\int_{\mathbb{R}^{\times}/\mathbb{Z}^{\times}}\varphi^{1}\left[\left(\begin{matrix}y&0&0\\0&1&0\\0&0&1\end{matrix}\right)\right]y^{s-1}d^{\times}y,\\
\widetilde{Z}(\widetilde{\varphi},s):&=\int_{\mathbb{R}^{\times}/\mathbb{Z}^{\times}}\int_{x\in\R}\widetilde{\varphi}^{1}\left[\left(\begin{matrix}y&0&0\\x&1&0\\0&0&1\end{matrix}\right)\cdot w'\right]y^{s-1}dxd^{\times}y.
\end{split}
\end{displaymath}
\end{definition}

\begin{definition}\label{The local Zeta integrals}(The local zeta integrals for $\GL_{3}(\R)$)\;\cite[\text{p.\;137}]{2}, \cite[\text{Theorem\;(11.2)}]{8}, \cite{15}\\
The two local zeta integrals $Z(W_{\varphi},s)$ and $\widetilde{Z}(\widetilde{W}_{\varphi},s)$ for $\GL_{3}(\R)$ are given by
\begin{displaymath}
\begin{split}
Z(W_{\varphi},s):&=\int_{\mathbb{R}^{\times}}W_{\varphi}\left[\left(\begin{matrix}y&0&0\\0&1&0\\0&0&1\end{matrix}\right)\right]|y|^{s-1}d^{\times}y,\\
\widetilde{Z}(\widetilde{W}_{\varphi},s):&=\int_{\mathbb{R}^{\times}}\int_{x\in\R}\widetilde{W}_{\varphi}\left[\left(\begin{matrix}y&0&0\\x&1&0\\0&0&1\end{matrix}\right)\cdot w'\right]|y|^{s-1}dxd^{\times}y.
\end{split}
\end{displaymath}
\end{definition}

\begin{definition}(The $L$-function for a representation $\pi$ on $\GL_{3}(\R)$)\;\cite[\text{Definition 6.5.2}]{4}, \cite{15}\\
We set
\begin{displaymath}
\begin{split}
L(\pi,s):&=\sum_{n=1}^{\infty}\frac{a(1,n)}{n^{s}}.
\end{split}
\end{displaymath}
\end{definition}

\begin{theorem}(The relation between $Z(\varphi,s)$ and $Z(W_{\varphi},s)$)\label{the relation}\;\cite[\text{(13.4.2)}]{8}, \cite[\text{(7.8)}]{15}\\
Let $W_{\varphi}$ be the Whittaker function corresponding to the automorphic function $\varphi\in\pi$.\\
We have the identity
\begin{displaymath}
\begin{split}
Z(\varphi,s)&=L(\pi,s)Z(W_{\varphi},s),
\end{split}
\end{displaymath}
which is fulfilled in the region of absolute convergence of $Z(W_{\varphi},s)$.
\end{theorem}

\begin{proof}
This is given in the adelic language in \cite{7,8}. We recall the basic unfolding calculation for completeness.\\
We calculate
\begin{displaymath}
\begin{split}
Z(\varphi,s)&=\int_{\mathbb{R}^{\times}/\mathbb{Z}^{\times}}\varphi^{1}\left[\left(\begin{matrix}y&0&0\\0&1&0\\0&0&1\end{matrix}\right)\right]y^{s-1}d^{\times}y=\int_{\mathbb{R}^{\times}/\mathbb{Z}^{\times}}\left(\sum_{\substack{n=-\infty\\n\neq0}}^{\infty}\frac{a(1,|n|)}{|n|}W_{\varphi}(a(ny))\right)y^{s-1}d^{\times}y\\
&=\sum_{\substack{n=-\infty\\n\neq0}}^{\infty}\frac{a(1,|n|)}{|n|}\int_{\mathbb{R}^{\times}/\mathbb{Z}^{\times}}W_{\varphi}(a(ny))y^{s-1}d^{\times}y=\sum_{n=1}^{\infty}\frac{a(1,n)}{n}\int_{\mathbb{R}^{\times}}W_{\varphi}(a(ny))|y|^{s-1}d^{\times}y\\
&=\sum_{n=1}^{\infty}\frac{a(1,n)}{n^{s}}\int_{\mathbb{R}^{\times}}W_{\varphi}(a(y))|y|^{s-1}d^{\times}y=L(\pi,s)Z(W_{\varphi},s).
\end{split}
\end{displaymath}
\end{proof}

\begin{theorem}\label{local functional equation}(The local functional equation for $\GL_{3}(\R)$)\;\cite[\text{(8.3)}]{2}, \cite[\text{Theorem\;(11.2)}]{8},\cite{16}\\
We have the local functional equation
\begin{displaymath}
\begin{split}
\widetilde{Z}(\widetilde{W}_{\varphi},1-s)&=\gamma(\pi,s)Z(W_{\varphi},s),
\end{split}
\end{displaymath}
where the gamma factor $\gamma(\pi,s)$ is defined by
\begin{displaymath}
\begin{split}
\gamma(\pi,s):&=\frac{L(\pi,s)}{L(\widetilde{\pi},1-s)}=\frac{\varepsilon_{\infty}(\pi,s)L_{\infty}(\widetilde{\pi},1-s)}{L_{\infty}(\pi,s)}=\pi^{3s-\frac{3}{2}}\frac{\Gamma\left(\frac{1-s+\alpha_{1}}{2}\right)\Gamma\left(\frac{1-s+\alpha_{2}}{2}\right)\Gamma\left(\frac{1-s+\alpha_{3}}{2}\right)}{\Gamma\left(\frac{s-\alpha_{1}}{2}\right)\Gamma\left(\frac{s-\alpha_{2}}{2}\right)\Gamma\left(\frac{s-\alpha_{3}}{2}\right)}.
\end{split}
\end{displaymath}
In the above equation, the three constants $\alpha_{1},\alpha_{2},\alpha_{3}\in\C$ are complex Langlands parameters, which satisfy $\left|\re(\alpha_{i})\right|<\frac{1}{2}$ for $i=1,2,3$ and depend on the representation $\pi$ of $\GL_{3}(\R)$.\\
Therefore, the gamma factor $\gamma(\pi,s)$ has no poles for $\re(s)\leq\frac{1}{2}$.
\end{theorem}

\begin{theorem}\label{Substructure of the Kirillov model of pi}(Substructure of the Kirillov model of $\pi$)\;\cite{10}, \cite[\text{Theorem 1}]{11}\\
Let $U_{2}(\R)$ be the subgroup of upper triangular unipotent matrices in $\GL_{2}(\R)$ with $1$'s on the diagonal and real entries above the diagonal and denote by $\theta:U_{2}(\R)\to\C$ the multiplicative character
\begin{displaymath}
\begin{split}
\theta:\left(\begin{matrix}1&x\\0&1\end{matrix}\right)&\mapsto e(x)\;\;\text{such that}\;\;\theta(u\cdot v)=\theta(u)\theta(v)\;\text{for all $u,v\in U_{2}(\R)$}.
\end{split}
\end{displaymath}
Let $\pi$ be a generic unitary irreducible representation of $\GL_{3}(\R)$ and denote by $C_{c}^{\infty}\left(\theta,\GL_{2}(\R)\right)$ the space of smooth and compactly supported modulo $U_{2}(\R)$ functions $f:\GL_{2}(\R)\to\C$ such that $f(ug)=\theta(u)f(g)$ for all $u\in U_{2}(\R)$, $g\in\GL_{2}(\R)$.\\
Given a function $\phi\in C_{c}^{\infty}\left(\theta,\GL_{2}(\R)\right)$ there is a unique Whittaker function $W\in\mathcal{K}(\pi)$ such that, for all $g\in\GL_{2}(\R)$,
\begin{displaymath}
\begin{split}
W\left[\left(\begin{matrix}g&0\\0&1\end{matrix}\right)\right]&=\phi(g).
\end{split}
\end{displaymath}
\end{theorem}

\section{The Geometric Approximate Functional Equation for $\GL_{3}(\R)$}
\label{The Geometric Approximate Functional Equation for GL_{3}(R)}

In this section, we construct a test vector $\varphi\in\pi$ coming from a carefully chosen element $W_{\varphi}$ in the Whittaker model of $\pi$, such that the local zeta integral $Z(W_{\varphi},\tfrac{1}{2}+s+iT)$ is of size $T^{3s/2-1/2}$ for all $s\in\C$ with $-\tfrac{1}{2}\leq\re(s)\leq\tfrac{1}{2}$ and $\im(s)\in\left[-\tfrac{3}{4}T,C_{1}T\right]$ for some fixed constant $C_{1}>0$.\\
Moreover, $\varphi$ enables us to write the global zeta integral $Z(\varphi,\tfrac{1}{2}+iT)$ for $\GL_{3}(\R)$ as a truncated global zeta integral with a small error term of size $O(T^{1/4-\kappa/2+\varepsilon})$ for a positive number $\kappa$.\\
We will call this truncated integral representation of the global zeta integral $Z(\varphi,\tfrac{1}{2}+iT)$ for $\GL_{3}(\R)$ the geometric approximate functional equation for $\GL_{3}(\R)$ as in \cite{14}.\\
The geometric approximate functional equation for $\GL_{2}(\R)$, namely
\begin{displaymath}
\begin{split}
Z(\varphi_{1},\tfrac{1}{2}+iT)&=\int_{y\in\mathbb{R}^{\times}/\mathbb{Z}^{\times}}\varphi_{1}(a(y))y^{iT}d^{\times}y\\
&=\int_{y\in\mathbb{R}_{+}^{\times}}\varphi_{1}(a(y))y^{iT}\left(h\left(\frac{y}{T^{\kappa}}\right)-h\left(\frac{y}{T^{-\kappa}}\right)\right)d^{\times}y+O\left(T^{-\kappa/2}\right),
\end{split}
\end{displaymath}
where
\begin{displaymath}
\begin{split}
\varphi_{1}(a(y)):&=\sum_{\substack{n=-\infty\\n\neq0}}^{\infty}\frac{a(|n|)}{\sqrt{|n|}}W_{\varphi_{1}}(a(ny))
\end{split}
\end{displaymath}
with
\begin{displaymath}
\begin{split}
W_{\varphi_{1}}(a(y))=W_{\varphi_{1}}\left[\left(\begin{matrix}y&0\\0&1\end{matrix}\right)\right]:=e(-y)W_{0}\left(\frac{y}{T}\right)\;\;\text{and}\;\;\kappa\geq0,
\end{split}
\end{displaymath}
is a consequence of \cite[\text{Lemma 5.1.4}]{14}. Here $h$ is the function defined before Theorem \ref{The geometric approximate functional equation for GL_3(R)} and $W_{0}$ is a compactly supported nonzero smooth function which is nonnegative.\\
There is a strong relation between the geometric approximate functional equation and the usual approximate functional equation, which we will discuss at the end of this section.\\

\subsection{Choice of the suitable Whittaker Function}\label{suitable Whittaker function}

In this subsection, we construct a Whittaker function $W_{\varphi}$ for $\GL_{3}(\R)$ whose $\GL(1)$ Mellin transform localizes to frequency of size approximately $T$.\\
\\
Let $C_{1}>0$ be a fixed constant. We fix once and for all a compactly supported nonzero smooth function $V_{0}\in C_{c}^{\infty}(\R_{+}^{\times})$ which is nonnegative and such that the support of $V_{0}$ is the interval $[\frac{1}{8\pi},\frac{1+C_{1}}{2\pi}]$, and let $T\in\R^{+}$ be a large positive real parameter. We choose the support of $V_{0}$ in this way, because $V_{0}$ should be nonzero on the interval $\left[\frac{1}{4\pi},\frac{1}{2\pi}\right]$ to obtain \eqref{stationary phase relation} later on. Let $V_{1}\in C_{c}^{\infty}(\R_{+}^{\times})$ be another fixed and compactly supported smooth function such that $V_{1}(1)=1$.\\
Let the function $\phi:\GL_{2}(\R)\to\C$ be defined by
\begin{displaymath}
\begin{split}
\phi\left[\left(\begin{matrix}y&x\\0&1\end{matrix}\right)\cdot O_{2}(\R)\cdot\left(\begin{matrix}z&0\\0&z\end{matrix}\right)\right]:&=T^{3/4}V_{0}\left(\frac{y}{T^{3/2}}\right)e\left(-\frac{y}{\sqrt{T}}\right)V_{1}(z)e(x)
\end{split}
\end{displaymath}
for $y>0$, $x\in\R$ and $z\in\R^{\times}$.\\
Then by Theorem \ref{Substructure of the Kirillov model of pi} and the $1$-periodicity of the function $\phi$ in the $x$-variable, there exists a unique Whittaker function $W_{\varphi}:\GL_{3}(\R)\to\C$ such that
\begin{displaymath}
\begin{split}
W_{\varphi}\left[\left(\begin{matrix}g&0\\0&1\end{matrix}\right)\right]&=\phi(g)\;\;\text{for all $g\in\GL_{2}(\R)$.}
\end{split}
\end{displaymath}
Therefore, this special Whittaker function $W_{\varphi}$ for $\GL_{3}(\R)$ satisfies
\begin{displaymath}
\begin{split}
W_{\varphi}\left[\left(\begin{matrix}\left(\begin{matrix}y&x\\0&1\end{matrix}\right)\cdot O_{2}(\R)\cdot\left(\begin{matrix}z&0\\0&z\end{matrix}\right)&0\\0&1\end{matrix}\right)\right]&=T^{3/4}V_{0}\left(\frac{y}{T^{3/2}}\right)e\left(-\frac{y}{\sqrt{T}}\right)V_{1}(z)e(x)
\end{split}
\end{displaymath}
for $y>0$, $x\in\R$ and $z\in\R^{\times}$.\\
Because the representation space $\pi$ is isomorphic to the Whittaker model $\mathcal{W}(\pi)$ attached to $\pi$, this Whittaker function $W_{\varphi}$ gives by the Fourier-Whittaker expansion rise to a vector $\varphi\in\pi$ inside the automorphic representation $\pi$.

\subsection{The Truncation of the global Zeta Integral for $\GL_{3}(\R)$}\label{stationary phase}

Let $W_{\varphi}$ be the Whittaker function constructed in \S\ref{suitable Whittaker function}. Let $s\in\C$ with $-\tfrac{1}{2}\leq\re(s)\leq\tfrac{1}{2}$ and $\im(s)\in\left[-\tfrac{3}{4}T,C_{1}T\right]$ for some fixed constant $C_{1}>0$. Setting the defining expression for $W_{\varphi}$ into the local zeta integral $Z(W_{\varphi},\tfrac{1}{2}+s+iT)$, we obtain
\begin{displaymath}
\begin{split}
Z(W_{\varphi},\tfrac{1}{2}+s+iT)&=\int_{\R_{+}^{\times}}T^{3/4}V_{0}\left(\frac{y}{T^{3/2}}\right)e\left(-\frac{y}{\sqrt{T}}\right)y^{iT-1/2+s}d^{\times}y,
\end{split}
\end{displaymath}
which transforms by the change of variable $y:=T^{3/2}z$ to
\begin{displaymath}
\begin{split}
Z(W_{\varphi},\tfrac{1}{2}+s+iT)&\ll T^{3s/2}\underbrace{\int_{\R_{+}^{\times}}V_{0}\left(z\right)e\left(-Tz+\frac{T+\im(s)}{2\pi}\log(z)\right)d^{\times}z}_{\ll T^{-1/2}}.
\end{split}
\end{displaymath}
In the above calculation, we have used the results of the stationary phase analysis in \cite[\text{Lemma 2.8}]{19}.\\
Therefore, we have the following stationary phase computation
\begin{equation}\label{local zeta integral bound}
\begin{split}
Z(W_{\varphi},\tfrac{1}{2}+s+iT)&\ll T^{3s/2-1/2}\;\;\text{as $T\rightarrow\infty$},
\end{split}
\end{equation}
where $s\in\C$ with $\re(s)\in[-\tfrac{1}{2},\tfrac{1}{2}]$ and $\im(s)\in\left[-\tfrac{3}{4}T,C_{1}T\right]$.\\
If $s\in\C$ with $-\tfrac{1}{2}\leq\re(s)\leq\tfrac{1}{2}$ and $\im(s)\notin\left[-\tfrac{3}{4}T,C_{1}T\right]$, then we have by \cite[\text{Lemma 2.6}]{19} that $Z(W_{\varphi},\frac{1}{2}+s+iT)$ is negligible, i.e., of size $O(T^{-N})$ for each fixed number $N$.\\
We fix once and for all a smooth even function $h\in C_{c}^{\infty}(\R)$ with values in $[0,1]$ that is identically $1$ in the interval $[-1,1]$ and falls off to zero outside of this interval, such that it is zero on $(-\infty,-2]$ and on $[2,\infty)$.

With this choice of the function $h$ and the Whittaker function $W_{\varphi}$, we obtain the following result.

\begin{theorem}\label{The geometric approximate functional equation for GL_3(R)}(The geometric approximate functional equation for $\GL_{3}(\R)$)\\
Fix $\kappa\geq0$. Let $\varphi\in\pi$ be the automorphic function corresponding to the Whittaker function $W_{\varphi}\in\mathcal{W}(\pi)$ constructed in \S\ref{suitable Whittaker function}.\\
We have
\begin{displaymath}
\begin{split}
Z(\varphi,\tfrac{1}{2}+iT)&=\int_{\R_{+}^{\times}}\varphi^{1}(a(y))\left(h\left(\frac{y}{T^{\kappa}}\right)-h\left(\frac{y}{T^{-\kappa}}\right)\right)y^{iT-1/2}d^{\times}y+O\left(T^{1/4-\kappa/2+\varepsilon}\right).
\end{split}
\end{displaymath}
\end{theorem}

\begin{proof}
Having made the above choices, the proof is essentially identical to that given in \cite[\text{Lemma 5.1.4, pp.\;254--256}]{14}. We record the proof for completeness.\\
We have
\begin{displaymath}
\begin{split}
\int_{\mathbb{R}_{+}^{\times}}\varphi^{1}(a(y))y^{iT-1/2}d^{\times}y&=\int_{\mathbb{R}_{+}^{\times}}\varphi^{1}(a(y))y^{iT-1/2}\left(h\left(\frac{y}{T^{\kappa}}\right)-h\left(\frac{y}{T^{-\kappa}}\right)\right)d^{\times}y\\
&\quad+\int_{\mathbb{R}_{+}^{\times}}\varphi^{1}(a(y))y^{iT-1/2}\left(1-h\left(\frac{y}{T^{\kappa}}\right)\right)d^{\times}y
\\
&\quad+\int_{\mathbb{R}_{+}^{\times}}\varphi^{1}(a(y))y^{iT-1/2}h\left(\frac{y}{T^{-\kappa}}\right)d^{\times}y.
\end{split}
\end{displaymath}
We estimate the two integrals
\begin{displaymath}
\begin{split}
\int_{\mathbb{R}_{+}^{\times}}\varphi^{1}(a(y))y^{iT-1/2}\left(1-h\left(\frac{y}{T^{\kappa}}\right)\right)d^{\times}y\;\;\;\;\text{and}\;\;\;\;\int_{\mathbb{R}_{+}^{\times}}\varphi^{1}(a(y))y^{iT-1/2}h\left(\frac{y}{T^{-\kappa}}\right)d^{\times}y
\end{split}
\end{displaymath}
separately and show that both integrals are $O\left(T^{1/4-\kappa/2+\varepsilon}\right)$.\\
\\
First, we estimate the second integral. By the Mellin inversion formula, we can write
\begin{displaymath}
\begin{split}
h\left(\frac{y}{T^{-\kappa}}\right)=\frac{1}{2\pi i}\int_{-c-i\infty}^{-c+i\infty}H(-s)T^{\kappa s}y^{s}ds\;\;\text{for any real number $c>0$},
\end{split}
\end{displaymath}
where
\begin{displaymath}
\begin{split}
H(s):=\int_{0}^{\infty}h(y)y^{s}d^{\times}y\;\;\text{for $s\in\C$ with $\re(s)>0$}
\end{split}
\end{displaymath}
is the Mellin transform of the function $h(y)$.\\
Substituting the above expression for $h\left(\frac{y}{T^{-\kappa}}\right)$ into the second integral $\int_{\mathbb{R}_{+}^{\times}}\varphi^{1}(a(y))y^{iT-1/2}h\left(\frac{y}{T^{-\kappa}}\right)d^{\times}y$ and interchanging the two integration processes, we calculate
\begin{displaymath}
\begin{split}
\int_{\mathbb{R}_{+}^{\times}}\varphi^{1}(a(y))y^{iT-1/2}h\left(\frac{y}{T^{-\kappa}}\right)d^{\times}y&=\frac{1}{2\pi i}\int_{-c-i\infty}^{-c+i\infty}H(-s)T^{\kappa s}\int_{\mathbb{R}_{+}^{\times}}\varphi^{1}(a(y))y^{s+iT-1/2}d^{\times}yds\\
&=\frac{1}{2\pi i}\int_{-c-i\infty}^{-c+i\infty}H(-s)T^{\kappa s}L(\pi,\tfrac{1}{2}+s+iT)Z(W_{\varphi},\tfrac{1}{2}+s+iT)ds
\end{split}
\end{displaymath}
by Theorem \ref{the relation}.\\
For $\re(s)=-\frac{1}{2}$, we can estimate
\begin{displaymath}
\begin{split}
L(\pi,\tfrac{1}{2}+s+iT)Z(W_{\varphi},\tfrac{1}{2}+s+iT)&\ll\left(1+|s|\right)^{3/2}T^{1/4+\varepsilon},
\end{split}
\end{displaymath}
where we have used the stationary phase analysis from the beginning of section \S\ref{stationary phase} and the convexity bound \cite[\text{(5.20)}]{6}.\\
Therefore, shifting the contour to $\re(s)=-\frac{1}{2}$, we see that
\begin{displaymath}
\begin{split}
\int_{\mathbb{R}_{+}^{\times}}\varphi^{1}(a(y))y^{iT-1/2}h\left(\frac{y}{T^{-\kappa}}\right)d^{\times}y&=\frac{1}{2\pi i}\int_{-c-i\infty}^{-c+i\infty}H(-s)T^{\kappa s}L(\pi,\tfrac{1}{2}+s+iT)Z(W_{\varphi},\tfrac{1}{2}+s+iT)ds\\
&\ll T^{1/4-\kappa/2+\varepsilon}.
\end{split}
\end{displaymath}
\\
A similar calculation shows that the first integral also satisfies
\begin{displaymath}
\begin{split}
\int_{\mathbb{R}_{+}^{\times}}\varphi^{1}(a(y))y^{iT-1/2}\left(1-h\left(\frac{y}{T^{\kappa}}\right)\right)d^{\times}y&\ll T^{1/4-\kappa/2+\varepsilon}
\end{split}
\end{displaymath}
and the theorem is proved.
\end{proof}

\subsection{From the Geometric Approximate Functional Equation to an Approximate Functional Equation for $Z(\varphi,\tfrac{1}{2}+iT)$}
\label{From the Geometric Approximate Functional Equation to an Approximate Functional Equation for Z(varphi,tfrac{1}{2}+iT)}

In the present subsection, we provide an approximate functional equation for $Z(\varphi,\tfrac{1}{2}+iT)$, which will be employed in the deduction of a subconvex bound for $\GL_{3}(\R)$ using the integral representation of the corresponding $L$-function. This shows the capabilities of our Whittaker function $W_{\varphi}$, which is able to imitate Lemma 3.3 of Lin's argument \cite[\text{p.\;1905}]{13}.\\
\begin{lemma}\label{the shape of the local zeta integral and its truncation}(The shape of the local zeta integral $Z(W_{\varphi},\frac{1}{2}+iT)$ and its truncation)\\
Let $W_{\varphi}\in\mathcal{W}(\pi)$ be the Whittaker function constructed in \S\ref{suitable Whittaker function}. Let $f\in C^{\infty}(\R^{\times})$ be a smooth function depending on $T$ such that $y^{j}f^{(j)}(y)\ll_{j}1$ for all $j\geq0$ as $T\rightarrow\infty$ and let $u\in C_{c}^{\infty}(\R_{+}^{\times})$ be a fixed compactly supported smooth function on $\R_{+}^{\times}$. Let $Y\in[T^{-\varepsilon},T^{\kappa}]$.\\
It holds for $n\in[T^{3/2-\kappa},T^{3/2+\varepsilon}]$ that
\begin{displaymath}
\begin{split}
\int_{\R_{+}^{\times}}W_{\varphi}(a(y))f\left(\frac{T^{3}}{4\pi^{2}ny}\right)y^{iT-1/2}d^{\times}y&=C_{T}\cdot T^{-1/2}\cdot f\left(\frac{T^{3/2}}{2\pi n}\right)+O(T^{-3/2})\;\;\text{as $T\rightarrow\infty$},
\end{split}
\end{displaymath}
as well as
\begin{displaymath}
\begin{split}
\int_{\R_{+}^{\times}}W_{\varphi}(a(y))f\left(\frac{y}{n}\right)u\left(\frac{y}{nY}\right)y^{iT-1/2}d^{\times}y&=C_{T}\cdot T^{-1/2}\cdot f\left(\frac{T^{3/2}}{2\pi n}\right)\cdot u\left(\frac{T^{3/2}}{2\pi nY}\right)+O(T^{-3/2})
\end{split}
\end{displaymath}
and in particular that
\begin{displaymath}
\begin{split}
Z(W_{\varphi},\tfrac{1}{2}+iT)&=C_{T}\cdot T^{-1/2}+O(T^{-3/2})\;\;\text{as $T\rightarrow\infty$},
\end{split}
\end{displaymath}
where $C_{T}$ is given by
\begin{displaymath}
\begin{split}
C_{T}:&=(2\pi)^{1-iT}e^{-\frac{\pi i}{4}}T^{\frac{3}{2}iT}e\left(-\frac{T}{2\pi}\right)V_{0}\left(\frac{1}{2\pi}\right).
\end{split}
\end{displaymath}
\end{lemma}

\begin{proof}
For the first identity, we use \cite[\text{Lemma 5}]{16} to calculate that
\begin{displaymath}
\begin{split}
&\int_{\R_{+}^{\times}}W_{\varphi}(a(y))f\left(\frac{T^{3}}{4\pi^{2}ny}\right)y^{iT-1/2}d^{\times}y\\
&=T^{\frac{3}{2}iT}\int_{\R_{+}^{\times}}V_{0}(z)e(-Tz)f\left(\frac{T^{3/2}}{4\pi^{2}nz}\right)z^{iT-1/2}d^{\times}z\\
&=(2\pi)^{1-iT}e^{-\frac{\pi i}{4}}T^{\frac{3}{2}iT}e\left(-\frac{T}{2\pi}\right)T^{-1/2}f\left(\frac{T^{3/2}}{2\pi n}\right)V_{0}\left(\frac{1}{2\pi}\right)+O(T^{-3/2}).
\end{split}
\end{displaymath}
By an extended analysis of the beginning of section \S\ref{stationary phase}, we get again via the stationary phase method \cite[\text{Lemma 5}]{16} that
\begin{displaymath}
\begin{split}
&\int_{\R_{+}^{\times}}W_{\varphi}(a(y))f\left(\frac{y}{n}\right)u\left(\frac{y}{nY}\right)y^{iT-1/2}d^{\times}y\\
&=T^{\frac{3}{2}iT}\int_{\R_{+}^{\times}}V_{0}(z)e(-Tz)f\left(\frac{T^{3/2}z}{n}\right)u\left(\frac{T^{3/2}z}{nY}\right)z^{iT-1/2}d^{\times}z\\
&=(2\pi)^{1-iT}e^{-\frac{\pi i}{4}}T^{\frac{3}{2}iT}e\left(-\frac{T}{2\pi}\right)T^{-1/2}f\left(\frac{T^{3/2}}{2\pi n}\right)u\left(\frac{T^{3/2}}{2\pi nY}\right)V_{0}\left(\frac{1}{2\pi}\right)+O(T^{-3/2}).
\end{split}
\end{displaymath}
This proves the second formula of Lemma \ref{the shape of the local zeta integral and its truncation}. To obtain the third identity, we pick $n$ fixed of size $T^{3/2}$, $Y$ of size $1$ and let $f(y)$ and $u(y)$ be two functions which are $=1$ on the support of $V_{0}$.
\end{proof}

\begin{lemma}(Integral to Sum and Sum to Integral Transformation)\label{Integral to Sum and Sum to Integral Transformation}\\
Let $f_{0}(y)$ be a smooth function on $\R_{+}^{\times}$ depending on $T$ such that $y^{j}f_{0}^{(j)}(y)\ll_{j}1$ for all $j\geq0$ and let $g\in C_{c}^{\infty}([\tfrac{1}{4\pi},\tfrac{1}{2\pi}])$ be a fixed compactly supported smooth function.\\
Let $\varphi\in\pi$ be the automorphic form corresponding to the Whittaker function $W_{\varphi}\in\mathcal{W}(\pi)$ constructed in \S\ref{suitable Whittaker function}.\\
We have for $Y\in[T^{-\varepsilon},T^{\kappa}]$ the transformation formula
\begin{displaymath}
\begin{split}
S_{f_{0}}(Y):&=\int_{\R_{+}^{\times}}\varphi^{1}(a(y))f_{0}(y)g\left(\frac{y}{Y}\right)y^{iT-1/2}d^{\times}y\\
&=C_{T}\cdot T^{-1/2}\sum_{n=1}^{\infty}\frac{a(1,n)}{n^{1/2+iT}}f_{0}\left(\frac{T^{3/2}}{2\pi n}\right)g\left(\frac{T^{3/2}}{2\pi nY}\right)+O(T^{-3/4+\varepsilon}).
\end{split}
\end{displaymath}
\end{lemma}

\begin{proof}
Expanding the above integral $S_{f_{0}}(Y)$ by employing the Fourier-Whittaker expansion of $\varphi^{1}(a(y))$ from equation \eqref{Fourier-Whittaker expansion}, we get that
\begin{displaymath}
\begin{split}
S_{f_{0}}(Y)&=\sum_{\substack{n=-\infty\\n\neq0}}^{\infty}\frac{a(1,|n|)}{|n|}\int_{y\in\R_{+}^{\times}}W_{\varphi}(a(ny))f_{0}(y)g\left(\frac{y}{Y}\right)y^{iT-1/2}d^{\times}y.
\end{split}
\end{displaymath}
The integral in the above expression restricts the $n$-sum to the range where $n\geq1$ and we get via a change of variable that
\begin{displaymath}
\begin{split}
S_{f_{0}}(Y)&=\sum_{n=1}^{\infty}\frac{a(1,n)}{n^{1/2+iT}}\int_{y\in\R_{+}^{\times}}W_{\varphi}(a(y))f_{0}\left(\frac{y}{n}\right)g\left(\frac{y}{nY}\right)y^{iT-1/2}d^{\times}y.
\end{split}
\end{displaymath}
The integral in this expression is equal to zero if $n\gg T^{3/2+\varepsilon}$, because $g\left(\frac{y}{nY}\right)$ is then constantly equal to zero as $y$ is of size $T^{3/2}$ and $g(y)$ is compactly supported.\\
Applying the second identity from Lemma \ref{the shape of the local zeta integral and its truncation} with the functions $f(y):=f_{0}(y)$ and $u(y):=g(y)$, we obtain by using the Rankin-Selberg estimate $\sum_{n\leq X}|a(1,n)|\ll X^{1+\varepsilon}$ \cite[\text{(4)}]{13} that
\begin{displaymath}
\begin{split}
S_{f_{0}}(Y)&=C_{T}\cdot T^{-1/2}\sum_{n=1}^{\infty}\frac{a(1,n)}{n^{1/2+iT}}f_{0}\left(\frac{T^{3/2}}{2\pi n}\right)g\left(\frac{T^{3/2}}{2\pi nY}\right)+O(T^{-3/4+\varepsilon}).
\end{split}
\end{displaymath}
\end{proof}

\noindent We obtain the following approximate functional equation.
\begin{theorem}\label{approximate equation}(An Approximate Functional Equation for $Z(\varphi,\tfrac{1}{2}+iT)$)\\
Let $c\geq0$. We define the smooth functions
\begin{displaymath}
\begin{split}
h_{0}(y):&=h\left(\frac{y}{T^{-\varepsilon}}\right)-h\left(\frac{y}{T^{-\kappa}}\right)\;\;\;\;\text{and}\;\;\;\;h_{1}(y):=h\left(\frac{y}{T^{\kappa}}\right)-h\left(\frac{y}{T^{-\varepsilon}}\right),\\
F(s):&=\int_{0}^{\infty}h_{0}(y)y^{s}d^{\times}y\;\;\;\;\text{and}\;\;\;\;G_{T}(z):=\frac{1}{2\pi i}\int_{-c-i\infty}^{-c+i\infty}F(-s)\left(\frac{T^{3}}{4\pi^{2}z}\right)^{s}\gamma(\pi,\tfrac{1}{2}+s+iT)ds,\\
h_{2}(z):&=\re\left(G_{T}(z)\right)\;\;\;\;\text{and}\;\;\;\;h_{3}(z):=\im\left(G_{T}(z)\right)
\end{split}
\end{displaymath}
and let $g\in C_{c}^{\infty}([\tfrac{1}{4\pi},\tfrac{1}{2\pi}])$ be a fixed compactly supported smooth function such that $g$ is positive on the interval $(\tfrac{1}{4\pi},\tfrac{1}{2\pi})$ and satisfies $\int_{\R_{+}^{\times}}g(y)d^{\times}y=1$.\\
It holds that $z^{j}h_{m}^{(j)}(z)\ll_{j}1$ for all $j\geq0$ and $m=1,2,3$.\\
We have for any $\kappa\in[0,\frac{3}{2}]$ and $\varphi\in\pi$ corresponding to the Whittaker function $W_{\varphi}\in\mathcal{W}(\pi)$ constructed in \S\ref{suitable Whittaker function} that
\begin{displaymath}
\begin{split}
Z(\varphi,\tfrac{1}{2}+iT)&\ll T^{\varepsilon}\sum_{f_{0}\in\left\{h_{1},h_{2},h_{3}\right\}}\sup_{T^{-\varepsilon}\leq Y\leq T^{\kappa}}\left\{\left|S_{f_{0}}(Y)\right|\right\}+O\left(T^{1/4-\kappa/2+\varepsilon}\right),
\end{split}
\end{displaymath}
where
\begin{displaymath}
\begin{split}
S_{f_{0}}(Y):&=\int_{\R_{+}^{\times}}\varphi^{1}(a(y))f_{0}(y)g\left(\frac{y}{Y}\right)y^{iT-1/2}d^{\times}y.
\end{split}
\end{displaymath}
\end{theorem}

\begin{proof}
Let $k(y):=h\left(\frac{y}{T^{\kappa}}\right)-h\left(\frac{y}{T^{-\kappa}}\right)$. By making a smooth dyadic subdivision of the truncated global zeta integral, we get by the geometric approximate functional equation that
\begin{displaymath}
\begin{split}
Z(\varphi,\tfrac{1}{2}+iT)&\ll T^{\varepsilon}\sup_{T^{-\kappa}\leq Y\leq T^{\kappa}}\left\{\int_{\R_{+}^{\times}}\varphi^{1}(a(y))k(y)g\left(\frac{y}{Y}\right)y^{iT-1/2}d^{\times}y\right\}+O\left(T^{1/4-\kappa/2+\varepsilon}\right).
\end{split}
\end{displaymath}
The error term $O\left(T^{1/4-\kappa/2+\varepsilon}\right)$ in this expression comes from the geometric approximate functional equation in Theorem \ref{The geometric approximate functional equation for GL_3(R)}.\\
The idea to truncate the $Y$-range after a smooth dyadic subdivision further to $T^{-\varepsilon}\leq Y\leq T^{\kappa}$ and to project the contribution in the supremum coming from the terms with $T^{-\kappa}\leq Y\leq T^{-\varepsilon}$ onto the terms with $T^{-\varepsilon}\leq Y\leq T^{\kappa}$ by using the local functional equation from Theorem \ref{local functional equation} was proposed to us by our advisor. This is also the strategy which we follow below.\\
Define the integrals
\begin{displaymath}
\begin{split}
I_{m}:&=\int_{\R_{+}^{\times}}\varphi^{1}(a(y))h_{m}(y)y^{iT-1/2}d^{\times}y\;\;\;\;\text{for $m=0,1,2,3$}.
\end{split}
\end{displaymath}
We start with the observation that
\begin{displaymath}
\begin{split}
\int_{\R_{+}^{\times}}\varphi^{1}(a(y))k(y)y^{iT-1/2}d^{\times}y&=I_{0}+I_{1}
\end{split}
\end{displaymath}
and make the change of variables $y\leftrightarrow\tfrac{1}{y}$ in the first integral $I_{0}$ to get that
\begin{displaymath}
\begin{split}
I_{0}&=\int_{y\in\R_{+}^{\times}}\varphi^{1}(a(1/y))h_{0}(1/y)y^{1/2-iT}d^{\times}y.
\end{split}
\end{displaymath}
Using the $\GL_{3}(\R)$ projection identity from Theorem \ref{projection identity}, the above transforms to
\begin{displaymath}
\begin{split}
I_{0}&=\int_{y\in\R_{+}^{\times}}\int_{x\in\R}\widetilde{\varphi}^{1}\left[\left(\begin{matrix}y&0&0\\xy&1&0\\0&0&1\end{matrix}\right)\cdot w'\right]h_{0}(1/y)y^{1/2-iT}dxd^{\times}y.
\end{split}
\end{displaymath}
Changing variables and using the Fourier-Whittaker expansion of $\widetilde{\varphi}^{1}(g)$ from \eqref{Dual form Whittaker expansion}, we obtain
\begin{displaymath}
\begin{split}
I_{0}&=\sum_{\substack{n=-\infty\\n\neq0}}^{\infty}\frac{a(|n|,1)}{|n|}\int_{y\in\R_{+}^{\times}}\int_{x\in\R}\widetilde{W}_{\varphi}\left[\left(\begin{matrix}ny&0&0\\x&1&0\\0&0&1\end{matrix}\right)\cdot w'\right]h_{0}(1/y)y^{-1/2-iT}dxd^{\times}y.
\end{split}
\end{displaymath}
Splitting the sum over $n$ into positive and negative contributions and changing variables again, we get that
\begin{displaymath}
\begin{split}
I_{0}&=\sum_{n=1}^{\infty}\frac{a(n,1)}{n}\int_{y\in\R^{\times}}\int_{x\in\R}\widetilde{W}_{\varphi}\left[\left(\begin{matrix}ny&0&0\\x&1&0\\0&0&1\end{matrix}\right)\cdot w'\right]h_{0}(1/|y|)|y|^{-1/2-iT}dxd^{\times}y\\
&=\sum_{n=1}^{\infty}\frac{a(n,1)}{n^{1/2-iT}}\int_{y\in\R^{\times}}\int_{x\in\R}\widetilde{W}_{\varphi}\left[\left(\begin{matrix}y&0&0\\x&1&0\\0&0&1\end{matrix}\right)\cdot w'\right]h_{0}(n/|y|)|y|^{-1/2-iT}dxd^{\times}y.
\end{split}
\end{displaymath}
Substituting into this expression the inverse Mellin transformation formula
\begin{displaymath}
\begin{split}
h_{0}\left(\frac{n}{|y|}\right)&=\frac{1}{2\pi i}\int_{-c-i\infty}^{-c+i\infty}F(-s)n^{s}|y|^{-s}ds\;\;\text{for any real number $c\in\R$}
\end{split}
\end{displaymath}
for the function $h_{0}\left(\frac{n}{|y|}\right)$ with $c\geq0$, we obtain
\begin{displaymath}
\begin{split}
I_{0}&=\sum_{n=1}^{\infty}\frac{a(n,1)}{n^{1/2-iT}}\frac{1}{2\pi i}\int_{-c-i\infty}^{-c+i\infty}F(-s)n^{s}\widetilde{Z}(\widetilde{W}_{\varphi},\tfrac{1}{2}-s-iT)ds.
\end{split}
\end{displaymath}
Using Theorem \ref{local functional equation} and interchanging the order of integration, we get that
\begin{displaymath}
\begin{split}
I_{0}&=\sum_{n=1}^{\infty}\frac{a(n,1)}{n^{1/2-iT}}\frac{1}{2\pi i}\int_{-c-i\infty}^{-c+i\infty}F(-s)n^{s}\gamma(\pi,\tfrac{1}{2}+s+iT)Z(W_{\varphi},\tfrac{1}{2}+s+iT)ds\\
&=\sum_{n=1}^{\infty}\frac{a(n,1)}{n^{1/2-iT}}\int_{y\in\R_{+}^{\times}}W_{\varphi}(a(y))G_{T}\hspace{-0.08cm}\left(\frac{T^{3}}{4\pi^{2}ny}\right)y^{iT-1/2}d^{\times}y,
\end{split}
\end{displaymath}
where the function $G_{T}(z)$ is defined as in the statement of Theorem \ref{approximate equation} such that
\begin{displaymath}
\begin{split}
G_{T}\hspace{-0.08cm}\left(\frac{T^{3}}{4\pi^{2}ny}\right)&=\frac{1}{2\pi i}\int_{-c-i\infty}^{-c+i\infty}F(-s)n^{s}\gamma(\pi,\tfrac{1}{2}+s+iT)y^{s}ds.
\end{split}
\end{displaymath}
We defined $G_{T}(z)$ such that the above identity holds and this will make the final formula as simple as possible. The condition $c\geq0$ is necessary and ensures that the above sums for $I_{0}$ converge absolutely, because $G_{T}(z)\ll T^{-N}$ for all $N\in\mathbb{N}$ and $z\in[0,T^{-\varepsilon}]$ by the remarks below.\\
We have for $m=2,3$ and $z\in[0,T^{-\varepsilon}]$ that $h_{m}(z)\ll T^{-N}$ for all $N\in\mathbb{N}$, which can be seen by shifting the contour of the definition for $G_{T}(z)$ towards minus infinity and using the fact that $\gamma(\pi,\tfrac{1}{2}+s+iT)\ll\left(1+\frac{|s|}{T}\right)^{-3s}T^{-3s}$ in each strip $\left\{s\in\C:-\infty<\re(s)\leq0\right\}$ \cite[\text{(5.115)}]{6}, \cite[\text{(1.20)}]{18}. Moreover, shifting the contour to $0\pm i\infty$, we get that $|h_{m}(z)|\ll C:=\int_{-i\infty}^{i\infty}|F(-s)|ds<\infty$ for $m=2,3$ and all $z\in\R_{+}^{\times}$, because there are no poles of $\gamma(\pi,\tfrac{1}{2}+s+iT)$ for $s\in\C$ with $\re(s)\leq0$. Similarly, it follows also that $z^{j}G_{T}^{(j)}(z)\ll_{j}1$ and that $z^{j}h_{m}^{(j)}(z)\ll_{j}1$ for all $j\geq0$ and $m=2,3$.\\
By the local functional equation from Theorem \ref{local functional equation} and equation \eqref{local zeta integral bound}, we obtain that
\begin{displaymath}
\begin{split}
\widetilde{Z}(\widetilde{W}_{\varphi},\tfrac{1}{2}-s-iT)&\ll T^{-3s/2-1/2}.
\end{split}
\end{displaymath}
Shifting the contour from the definition of $G_{T}(z)$ to $0\pm i\infty$ and using Theorem \ref{local functional equation} to transform back the integral such that $\widetilde{Z}(\widetilde{W}_{\varphi},\tfrac{1}{2}-s-iT)$ reappears in it, we therefore see that
\begin{equation}\label{Whittaker function integral bound}
\begin{split}
\int_{y\in\R_{+}^{\times}}W_{\varphi}(a(y))G_{T}\hspace{-0.08cm}\left(\frac{T^{3}}{4\pi^{2}ny}\right)y^{iT-1/2}d^{\times}y&=\frac{1}{2\pi i}\int_{-i\infty}^{i\infty}F(-s)n^{s}\widetilde{Z}(\widetilde{W}_{\varphi},\tfrac{1}{2}-s-iT)ds\ll T^{-1/2}.
\end{split}
\end{equation}
Defining the two sums
\begin{displaymath}
\begin{split}
S_{m}:&=\sum_{n=1}^{\infty}\frac{a(n,1)}{n^{1/2-iT}}\int_{y\in\R_{+}^{\times}}W_{\varphi}(a(y))h_{m}\left(\frac{T^{3}}{4\pi^{2}ny}\right)y^{iT-1/2}d^{\times}y\;\;\;\;\text{for $m=2,3$},
\end{split}
\end{displaymath}
we get that
\begin{displaymath}
\begin{split}
I_{0}&=S_{2}+iS_{3}.
\end{split}
\end{displaymath}
By a smooth dyadic subdivision with the function $g(y)$, we get by using the above remarks on $h_{m}(z)$ for $m=2,3$ and the construction of $W_{\varphi}(a(y))$ in \S\ref{suitable Whittaker function} that
\begin{displaymath}
\begin{split}
S_{m}&\ll T^{\varepsilon}\sup_{T^{-\varepsilon}\leq Y\leq T^{3/2}}\left\{\left|D_{Y}\right|\right\}+O(T^{-A})\;\;\text{with $A>0$},
\end{split}
\end{displaymath}
where
\begin{displaymath}
\begin{split}
D_{Y}:&=\sum_{n=1}^{\infty}\frac{a(n,1)}{n^{1/2-iT}}g\left(\frac{T^{3/2}}{2\pi nY}\right)\int_{y\in\R_{+}^{\times}}W_{\varphi}(a(y))h_{m}\left(\frac{T^{3}}{4\pi^{2}ny}\right)y^{iT-1/2}d^{\times}y.
\end{split}
\end{displaymath}
The above bound for $S_{m}$ with $m=2,3$ can, by the Rankin-Selberg bound $\sum_{n\leq X}|a(n,1)|\ll X^{1+\varepsilon}$ \cite[\text{(4)}]{13} and because of \eqref{Whittaker function integral bound}, be further simplified to
\begin{displaymath}
\begin{split}
S_{m}&\ll T^{\varepsilon}\sup_{T^{-\varepsilon}\leq Y\leq T^{\kappa}}\left\{\left|D_{Y}\right|\right\}+O(T^{1/4-\kappa/2+\varepsilon}).
\end{split}
\end{displaymath}
Using the first identity of Lemma \ref{the shape of the local zeta integral and its truncation} with $f(y):=h_{m}(y)$, the above two suprema for $m=2,3$ are seen to be bounded by
\begin{displaymath}
\begin{split}
S_{m}&\ll T^{\varepsilon}\sup_{T^{-\varepsilon}\leq Y\leq T^{\kappa}}\left\{\left|C_{T}\cdot T^{-1/2}\sum_{n=1}^{\infty}\frac{a(n,1)}{n^{1/2-iT}}g\left(\frac{T^{3/2}}{2\pi nY}\right)h_{m}\left(\frac{T^{3/2}}{2\pi n}\right)\right|\right\}+O(T^{1/4-\kappa/2+\varepsilon}),
\end{split}
\end{displaymath}
because the error term $O(T^{1/4-\kappa/2+\varepsilon})$ dominates the other error term $O(T^{-3/4+\varepsilon})$ coming from Lemma \ref{the shape of the local zeta integral and its truncation}.\\
Moreover, by using that $a(n,1)=\overline{a(1,n)}$, $n^{1/2-iT}=\overline{n^{1/2+iT}}$ and that the $h_{m}(y)$'s, as well as $g(y)$ are real valued functions, we conclude that for $m=2,3$
\begin{displaymath}
\begin{split}
\overline{S_{m}}&\ll T^{\varepsilon}\sup_{T^{-\varepsilon}\leq Y\leq T^{\kappa}}\left\{\left|\overline{C_{T}}\cdot T^{-1/2}\sum_{n=1}^{\infty}\frac{a(1,n)}{n^{1/2+iT}}h_{m}\left(\frac{T^{3/2}}{2\pi n}\right)g\left(\frac{T^{3/2}}{2\pi nY}\right)\right|\right\}+O(T^{1/4-\kappa/2+\varepsilon}).
\end{split}
\end{displaymath}
From this expression, we conclude via Lemma \ref{Integral to Sum and Sum to Integral Transformation} and $f_{0}(y):=h_{m}(y)$ that for $m=2,3$, we have
\begin{displaymath}
\begin{split}
S_{m}&\ll T^{\varepsilon}\sup_{T^{-\varepsilon}\leq Y\leq T^{\kappa}}\left\{\left|\int_{\R_{+}^{\times}}\varphi^{1}(a(y))h_{m}(y)g\left(\frac{y}{Y}\right)y^{iT-1/2}d^{\times}y\right|\right\}+O(T^{1/4-\kappa/2+\varepsilon})
\end{split}
\end{displaymath}
to conclude finally by a smooth dyadic subdivision of the integral $I_{1}$ with the same function $g(y)$ that
\begin{displaymath}
\begin{split}
Z(\varphi,\tfrac{1}{2}+iT)&\ll T^{\varepsilon}\sup_{T^{-\varepsilon}\leq Y\leq T^{\kappa}}\left\{\left|\int_{\R_{+}^{\times}}\varphi^{1}(a(y))h_{1}(y)g\left(\frac{y}{Y}\right)y^{iT-1/2}d^{\times}y\right|\right\}+O(T^{1/4-\kappa/2+\varepsilon})\\
&\quad+T^{\varepsilon}\sum_{f_{0}\in\left\{h_{2},h_{3}\right\}}\sup_{T^{-\varepsilon}\leq Y\leq T^{\kappa}}\left\{\left|\int_{\R_{+}^{\times}}\varphi^{1}(a(y))f_{0}(y)g\left(\frac{y}{Y}\right)y^{iT-1/2}d^{\times}y\right|\right\}.
\end{split}
\end{displaymath}
This is the claimed formula.
\end{proof}

\begin{remark}Taking the supremum over the full range $T^{-\kappa}\leq Y\leq T^{\kappa}$ in Theorem \ref{approximate equation} is also enough to obtain a subconvex bound, because we would get a saving of $\frac{1}{60}$. Shrinking the range to $T^{-\varepsilon}\leq Y\leq T^{\kappa}$ has only the effect of optimizing the saving to Lin's $\frac{1}{36}$ \cite{13}.
\end{remark}

\section{Subconvexity for $\GL_{3}(\R)$ $L$-Functions via Integral Representations}
\label{Subconvexity for GL_{3}(R) L-Functions via Integral Representations}

Let $W_{\varphi}$ be the Whittaker function constructed in \S\ref{suitable Whittaker function} and $\varphi\in\pi$ the corresponding automorphic form.

\subsection{The Analysis of the Integral $S_{f_{0}}(Y)$}\label{integral analysis}

Let $T^{-\varepsilon}\leq Y\leq T^{\kappa}$ and let $f_{0}\in\left\{h_{1},h_{2},h_{3}\right\}$.\\
We have by the definition of $S_{f_{0}}(Y)$ and the Fourier-Whittaker expansion of $\varphi^{1}(a(y))$ that
\begin{displaymath}
\begin{split}
S_{f_{0}}(Y)&=\int_{\R_{+}^{\times}}\varphi^{1}(a(y))f_{0}(y)g\left(\frac{y}{Y}\right)y^{iT-1/2}d^{\times}y\\
&=\sum_{\substack{n=-\infty\\n\neq0}}^{\infty}\frac{a(1,|n|)}{|n|}\int_{\R_{+}^{\times}}W_{\varphi}(a(ny))f_{0}(y)g\left(\frac{y}{Y}\right)y^{iT-1/2}d^{\times}y.
\end{split}
\end{displaymath}
Setting into this expression the definition of the suitable Whittaker function $W_{\varphi}$ in \S\ref{suitable Whittaker function}, we deduce that
\begin{displaymath}
\begin{split}
S_{f_{0}}(Y)&=T^{3/4}\sum_{n=1}^{\infty}\frac{a(1,n)}{n}\int_{0}^{\infty}e\left(-\frac{ny}{\sqrt{T}}\right)y^{iT}V_{0}\left(\frac{ny}{T^{3/2}}\right)g\left(\frac{y}{Y}\right)\frac{f_{0}(y)}{\sqrt{y}}d^{\times}y.
\end{split}
\end{displaymath}
After performing the change of variables $y\leftrightarrow Yy$, we get
\begin{displaymath}
\begin{split}
S_{f_{0}}(Y)&=\frac{T^{3/4}}{Y^{1/2-iT}}\sum_{n=1}^{\infty}\frac{a(1,n)}{n}\int_{0}^{\infty}e\left(-\frac{nYy}{\sqrt{T}}\right)y^{iT}V_{0}\left(\frac{nYy}{T^{3/2}}\right)g(y)\frac{f_{0}(Yy)}{\sqrt{y}}d^{\times}y.
\end{split}
\end{displaymath}
Define the variable $N:=T^{3/2}/Y$, set $S(N):=S_{f_{0}}(Y)$ and let $f_{n}(y):=V_{0}\left(\frac{ny}{N}\right)g(y)\frac{f_{0}(Yy)}{y\sqrt{y}}$ for $n\geq1$ to get that
\begin{displaymath}
\begin{split}
S(N)&\ll\sqrt{N}\sum_{n=1}^{\infty}\frac{a(1,n)}{n}\int_{0}^{\infty}y^{iT}e\left(-\frac{nTy}{N}\right)f_{n}(y)dy.
\end{split}
\end{displaymath}
After the change of variables $y:=\frac{1}{x}$ and the definition $V_{n}(x):=\frac{f_{n}(\frac{1}{x})}{x^{2}}=V_{0}\left(\frac{n}{Nx}\right)g\left(\frac{1}{x}\right)f_{0}\left(\frac{Y}{x}\right)\frac{1}{\sqrt{x}}$, this is equal to
\begin{displaymath}
\begin{split}
S(N)&\ll\sqrt{N}\sum_{n=1}^{\infty}\frac{a(1,n)}{n}\int_{\R_{+}^{\times}}x^{-iT}e\left(-\frac{nT}{Nx}\right)V_{n}(x)dx.
\end{split}
\end{displaymath}
Set $V(x):=V_{0}\left(\frac{1}{x}\right)f_{0}\left(\frac{Y}{x}\right)\frac{1}{\sqrt{x}}$ and note that $V^{(j)}(x)\ll_{j}1$ for all $j\geq0$. By the stationary phase method \cite[\text{Lemma 5}]{16} and the bound $V_{n}^{(j)}(x)\ll_{j}1$ for all $j\geq0$ there exists a smooth and compactly supported function, for example $w_{0}(z):=\frac{V_{0}\left(\frac{1}{2\pi}\right)}{V_{0}\left(\frac{1}{2\pi z}\right)}g\left(\frac{1}{2\pi z}\right)\in C_{c}^{\infty}([1,2])$, such that
\begin{equation}\label{stationary phase relation}
\begin{split}
\int_{\R_{+}^{\times}}x^{-iT}e\left(-\frac{nT}{Nx}\right)V_{n}(x)dx
&=w_{0}\left(\frac{n}{N}\right)\int_{\R_{+}^{\times}}x^{-iT}e\left(-\frac{nT}{Nx}\right)V(x)dx+O(T^{-3/2}),
\end{split}
\end{equation}
because we have
\begin{displaymath}
\begin{split}
\int_{\R_{+}^{\times}}x^{-iT}e\left(-\frac{nT}{Nx}\right)V_{n}(x)dx
&=c_{T}\cdot T^{-1/2}V_{0}\left(\frac{1}{2\pi}\right)g\left(\frac{N}{2\pi n}\right)f_{0}\left(\frac{NY}{2\pi n}\right)\sqrt\frac{N}{2\pi n}+O(T^{-3/2}),\\
\int_{\R_{+}^{\times}}x^{-iT}e\left(-\frac{nT}{Nx}\right)V(x)dx&=c_{T}\cdot T^{-1/2}V_{0}\left(\frac{N}{2\pi n}\right)f_{0}\left(\frac{NY}{2\pi n}\right)\sqrt\frac{N}{2\pi n}+O(T^{-3/2}),
\end{split}
\end{displaymath}
where $c_{T}\in\C$ is given by
\begin{displaymath}
\begin{split}
c_{T}:=\sqrt{2\pi}e^{-\frac{\pi i}{4}}e\left(-\frac{T}{2\pi}\right)\left(\frac{2\pi n}{N}\right)^{1-iT}.
\end{split}
\end{displaymath}
In the above two asymptotic formulas, the implied constant of the respective error term is uniform, because it depends continuously on $\tfrac{n}{N}$, which varies in the compact set $[1,2]$.\\
This implies that
\begin{displaymath}
\begin{split}
S(N)&\ll\sqrt{N}\sum_{n=1}^{\infty}\frac{a(1,n)}{n}w_{0}\left(\frac{n}{N}\right)\int_{\R_{+}^{\times}}x^{-iT}e\left(-\frac{nT}{Nx}\right)V(x)dx.
\end{split}
\end{displaymath}
Absorbing the fraction $\frac{1}{n}$ into the weight function $w_{0}$ by defining $w(x):=\frac{w_{0}(x)}{x}\in C_{c}^{\infty}([1,2])$, we deduce that
\begin{displaymath}
\begin{split}
S(N)&\ll\frac{1}{\sqrt{N}}\sum_{n=1}^{\infty}a(1,n)w\left(\frac{n}{N}\right)\int_{\R_{+}^{\times}}x^{-iT}e\left(-\frac{nT}{Nx}\right)V(x)dx.
\end{split}
\end{displaymath}
We have therefore to study the main integral
\begin{displaymath}
\begin{split}
\int_{\R_{+}^{\times}}x^{-iT}e\left(-\frac{nT}{Nx}\right)V(x)dx,
\end{split}
\end{displaymath}
which is exactly the integral appearing in \cite[\text{pp.\;1908--1909}]{13} in the proof of Lin's key identity. This key observation tells us that Lin's key identity can be understood as replacing the above integral $\int_{\R_{+}^{\times}}x^{-iT}e\left(-\frac{nT}{Nx}\right)V(x)dx$ by the corresponding Riemann lattice sum plus its oscillations (error terms) around the exact value of the integral. This also explains why there are two terms, namely $\mathcal{F}$ and $\mathcal{O}$ present in this method. We will follow closely the work \cite{13} in the end of our argument.

\subsection{The Discretization and Amplification of the Main Integral}\label{The Discretization and Amplification of the Main Integral}

We use the letters $p$ and $\ell$ to denote prime numbers. Let $P$ and $L$ be two large parameters, which will be specified later as small powers of the parameter $T\in\R^{+}$. The notations $p\sim P$ and $\ell\sim L$ are used to denote prime numbers in the two dyadic segments $[P,2P]$ and $[L,2L]$ respectively. We also assume that $[P,2P]\cap[L,2L]=\emptyset$. The sums $\sum_{p\sim P}$ and $\sum_{\ell\sim L}$ describe sums over all the prime numbers $p\in[P,2P]$ and $\ell\in[L,2L]$.\\
We have the following
\begin{lemma}\label{keyidentity}(Another form of Lin's key identity)\;\cite{13}\\
We have
\begin{displaymath}
\begin{split}
\int_{\R_{+}^{\times}}x^{-iT}e\left(-\frac{nT}{Nx}\right)V(x)dx&=\left(\frac{\ell T}{Np}\right)^{1-iT}\sum_{r=1}^{\infty}r^{-iT}e\left(-\frac{np}{\ell r}\right)V\left(\frac{r}{Np/\ell T}\right)-\sum_{\substack{r\in\Z\\r\neq0}}\mathcal{J}_{iT}\left(n,\frac{rp}{\ell}\right)
\end{split}
\end{displaymath}
with
\begin{displaymath}
\begin{split}
\mathcal{J}_{iT}\left(n,\frac{rp}{\ell}\right):&=\int_{\R_{+}^{\times}}x^{-iT}e\left(-\frac{nT}{Nx}\right)V(x)e\left(-\frac{rNp}{\ell T}x\right)dx.
\end{split}
\end{displaymath}
\end{lemma}

\begin{proof}
We can calculate using the Poisson summation formula that
\begin{displaymath}
\begin{split}
\sum_{r=1}^{\infty}r^{-iT}e\left(-\frac{np}{\ell r}\right)V\left(\frac{r}{Np/\ell T}\right)&=\int_{\R_{+}^{\times}}z^{-iT}e\left(-\frac{np}{\ell z}\right)V\left(\frac{z}{Np/\ell T}\right)dz\\
&\quad+\sum_{\substack{r\in\Z\\r\neq0}}\int_{\R_{+}^{\times}}z^{-iT}e\left(-\frac{np}{\ell z}\right)V\left(\frac{z}{Np/\ell T}\right)e(-rz)dz.
\end{split}
\end{displaymath}
Making in the above expression the change of variables $z:=\frac{Np}{\ell T}x$, we deduce that
\begin{displaymath}
\begin{split}
\sum_{r=1}^{\infty}r^{-iT}e\left(-\frac{np}{\ell r}\right)V\left(\frac{r}{Np/\ell T}\right)&=\left(\frac{Np}{\ell T}\right)^{1-iT}\int_{\R_{+}^{\times}}x^{-iT}e\left(-\frac{nT}{Nx}\right)V(x)dx\\
&\quad+\left(\frac{Np}{\ell T}\right)^{1-iT}\sum_{\substack{r\in\Z\\r\neq0}}\int_{\R_{+}^{\times}}x^{-iT}e\left(-\frac{nT}{Nx}\right)V(x)e\left(-\frac{rNp}{\ell T}x\right)dx.
\end{split}
\end{displaymath}
Solving the above expression for $\int_{\R_{+}^{\times}}x^{-iT}e\left(-\frac{nT}{Nx}\right)V(x)dx$ implies the claimed identity.
\end{proof}

\noindent This implies the following
\begin{lemma}\label{amplifiedkeyidentity}(The amplified key identity)\;\cite{13}\\
We have
\begin{displaymath}
\begin{split}
\int_{\R_{+}^{\times}}x^{-iT}e\left(-\frac{nT}{Nx}\right)V(x)dx&\asymp\frac{T^{1+\varepsilon}}{NP^{2}}\sum_{p\sim P}p^{iT}\sum_{\ell\sim L}\ell^{-iT}\sum_{r=1}^{\infty}r^{-iT}e\left(-\frac{np}{\ell r}\right)V\left(\frac{r}{Np/\ell T}\right)\\
&\quad-\frac{T^{\varepsilon}}{PL}\sum_{p\sim P}\sum_{\ell\sim L}\sum_{\substack{r\in\Z\\r\neq0}}\mathcal{J}_{iT}\left(n,\frac{rp}{\ell}\right).
\end{split}
\end{displaymath}
\end{lemma}

\begin{proof}
Using the above Lemma \ref{keyidentity} and the identity
\begin{displaymath}
\begin{split}
\frac{\log(P)\log(L)}{PL}\sum_{p\sim P}\sum_{\ell\sim L}1&\asymp1,
\end{split}
\end{displaymath}
which is a direct consequence of the prime number theorem \cite{13}, we can calculate that
\begin{displaymath}
\begin{split}
\int_{\R_{+}^{\times}}x^{-iT}e\left(-\frac{nT}{Nx}\right)V(x)dx&\asymp\frac{\log(P)\log(L)}{PL}\sum_{p\sim P}\sum_{\ell\sim L}\int_{\R_{+}^{\times}}x^{-iT}e\left(-\frac{nT}{Nx}\right)V(x)dx\\
&=\frac{\log(P)\log(L)}{PL}\sum_{p\sim P}\sum_{\ell\sim L}\Bigg[\left(\frac{\ell T}{Np}\right)^{1-iT}\sum_{r=1}^{\infty}r^{-iT}e\left(-\frac{np}{\ell r}\right)V\left(\frac{r}{Np/\ell T}\right)\\
&\hspace{4.0cm}-\sum_{\substack{r\in\Z\\r\neq0}}\int_{\R_{+}^{\times}}x^{-iT}e\left(-\frac{nT}{Nx}\right)V(x)e\left(-\frac{rNp}{\ell T}x\right)dx\Bigg].
\end{split}
\end{displaymath}
This is the claimed formula.
\end{proof}

\subsection{The Final Bound for $L(\pi,\tfrac{1}{2}+iT)$}

Setting $\kappa:=\frac{1}{18}$ and the two variables $P$ and $L$ as in \cite{13} to
\begin{displaymath}
\begin{split}
P:&=T^{5/18}\;\;\text{and}\;\;L:=T^{1/9},
\end{split}
\end{displaymath}
we obtain using Theorem \ref{approximate equation} with $S_{f_{0}}(Y)=S(N)$ and following \cite{13} that
\begin{displaymath}
\begin{split}
Z(\varphi,\tfrac{1}{2}+iT)&\ll T^{\varepsilon}\sup_{T^{3/2-\kappa}\leq N\leq T^{3/2+\varepsilon}}\left\{\left|S(N)\right|\right\}+T^{1/4-\kappa/2+\varepsilon}\\
&\ll\left(\frac{T^{3/2+\varepsilon}P}{T^{3/2}L^{1/2}}+T^{3/8+\varepsilon}\left(\frac{PL}{T}\right)^{1/4}\right)+\left(\frac{T^{1/2+\varepsilon}}{P}+\frac{T^{1+\kappa/2+\varepsilon}L}{T^{3/4}P}\right)+T^{1/4-\kappa/2+\varepsilon}\\
&\ll T^{1/4-1/36+\varepsilon}.
\end{split}
\end{displaymath}
In this calculation, we have used that the above bound for $S(N)=S_{f_{0}}(Y)$ is independent of the function $f_{0}\in\left\{h_{1},h_{2},h_{3}\right\}$.\\
Finally, because it holds according to Lemma \ref{the shape of the local zeta integral and its truncation} that $Z(W_{\varphi},\tfrac{1}{2}+iT)=C_{T}\cdot T^{-1/2}+O(T^{-3/2})$, we have by Theorem \ref{the relation} that
\begin{displaymath}
\begin{split}
L(\pi,\tfrac{1}{2}+iT)&=\frac{Z(\varphi,\tfrac{1}{2}+iT)}{Z(W_{\varphi},\tfrac{1}{2}+iT)}\ll Z(\varphi,\tfrac{1}{2}+iT)T^{1/2}\ll T^{3/4-1/36+\varepsilon}.
\end{split}
\end{displaymath}
The saving $\frac{1}{36}$ is not the best currently known, because Munshi \cite{16} obtained a saving of $\frac{1}{16}$ and Aggarwal \cite{1} got a saving of $\frac{3}{40}$.

\section{Acknowledgement}
\label{Acknowledgement}

We would like to thank our advisor Professor Dr. Paul Nelson for entrusting us with this interesting project. He introduced us carefully to the representation theory of automorphic forms and the geometric approximate functional equation for $\GL_{2}(\R)$, explaining it in detail with all the proofs over $\R$. He also suggested to us a possible connection of the key identity with automorphic periods and representation theory. Beginning from a well chosen starting point, we gained very interesting insights into the key identity of the $\GL_{3}(\R)$-problem and the various findings of our research have become a unity.\\
Many thanks to Prof. Dr. Paul Nelson as well as to Prof. Dr. Roman Holowinsky, Prof. Dr. Matthew Young, Dr. Yongxiao Lin and the referee for their valuable suggestions which improved this paper.\\
\\
This work was supported by SNSF (Swiss National Science Foundation) under grant 169247.

\bibliography{mybibfile}

\begin{thebibliography}{9}

\bibitem{1}
Keshav Aggarwal, A new subconvex bound for $\GL(3)$ $L$-functions in the $t$-aspect, Int. J. Number Theory (2021), \url{https://arxiv.org/pdf/1903.09638.pdf}.

\bibitem{2}
Daniel Bump, \emph{Automorphic Forms on $\GL(3,\R)$}, Lecture Notes in Mathematics, Springer, Berlin Heidelberg New York Tokyo 1984.

\bibitem{3}
Daniel Bump, \emph{Number theory, trace formulas and discrete groups}, (Oslo, 1987), Academic Press, Boston, MA, 1989, Chapter 4.

\bibitem{4}
Dorian Goldfeld, \emph{Automorphic Forms and L-Functions for the Group $\GL(n;\R)$}, Cambridge Stud. Adv. Math. 99, Cambridge Univ. Press, Cambridge, 2006, with an appendix by Kevin A. Broughan. MR 2254662.

\bibitem{5}
Roman Holowinsky and Paul D. Nelson, Subconvex bounds on $\GL_{3}$ via degeneration to frequency zero, \emph{Math. Ann.} (2018), Springer Nature, 372:\;299--319, 
https://doi.org/10.1007/s00208-018-1711-y, 299--319.

\bibitem{6}
Henryk Iwaniec and Emmanuel Kowalski, \emph{Analytic number theory}, American Mathematical Society Colloquium Publications, vol. 53, American Mathematical Society, Providence, RI, 2004, MR 20612149.

\bibitem{7}
Herv\'e Jacquet, Ilja Iosifovitch Piatetski-Shapiro and Joseph Shalika, Automorphic Forms on $\GL(3)$ I, \emph{Ann. of Math.}, Second Series, Vol. 109, No. 1 (Jan., 1979), pp. 169--212.

\bibitem{8}
Herv\'e Jacquet, Ilja Iosifovitch Piatetski-Shapiro and Joseph Shalika, Automorphic Forms on $\GL(3)$ II, \emph{Ann. of Math.}, Second Series, Vol. 109, No. 2 (May, 1979), pp. 213--258.

\bibitem{9}
Herv\'e Jacquet, Dirichlet Series for the Group $\GL(N)$, G. H. Iwasawa et al., \emph{Automorphic Forms, Representation Theory and Arithmetic}, Springer-Verlag Berlin Heidelberg 1981.

\bibitem{10}
Herv\'e Jacquet, Distinction by the quasi-split unitary group, \emph{Isr. J. Math.} 178 (1) (2010), 269--324.

\bibitem{11}
Alexander Kemarsky, A note on the Kirillov model for representations of $\GL_{n}(\C)$, \emph{Lie algebras/Functional analysis}, ScienceDirect, Elsevier, C. R. Acad. Sci. Paris, Ser. I 353 (2015), 579--582.

\bibitem{12}
Emmanuel Kowalski, Yongxiao Lin, Philippe Michel, and Will Sawin, Periodic twists of $\GL_{3}$-automorphic forms, (English summary) \emph{Forum Math. Sigma} 8 (2020), Paper No. e15, 39 pp., 11F55, \url{https://arxiv.org/pdf/1905.05080v3.pdf}.

\bibitem{13}
Yongxiao Lin, Bounds for twists of $\GL(3)$ $L$-functions, \emph{J. Eur. Math. Soc. (JEMS)} 23 (2021), no. 6, 1899--1924, \url{https://ems.press/content/serial-article-files/12165}.

\bibitem{14}
Philippe Michel and Akshay Venkatesh, The Subconvexity Problem for $\GL_{2}$, \emph{Publ. Math. Inst. Hautes \'Etudes Sci.} (2010), no. 111, 171--271. MR 2653249.

\bibitem{15}
Stephen D. Miller, Cancellation in additively twisted sums on $\GL(n)$. (English summary) \emph{Amer. J. Math. 128} (2006), no. 3, 699--729. 11F67 (11F70).

\bibitem{16}
Ritabrata Munshi, The circle method and bounds for $L$-functions--III: t-aspect subconvexity for $\GL(3)$ $L$-functions, \emph{J. Amer. Math. Soc. 28} (2015), no. 4, MR 3369905, 913--938.

\bibitem{17}
Ritabrata Munshi, The circle method and bounds for $L$-functions--IV: Subconvexity for twists of $\GL(3)$ $L$-functions, \emph{Ann. of Math. {\bf182}} (2015), 617--672,\\
\url{http://annals.math.princeton.edu/wp-content/uploads/Munshi.pdf}.

\bibitem{18}
Paul D. Nelson, Eisenstein Series and the Cubic Moment for $\PGL_{2}$, arXiv e-prints (2019), \url{https://arxiv.org/pdf/1911.06310v3.pdf}.

\bibitem{19}
Terence Tao, Lecture Notes 8 for 247B, \url{https://www.math.ucla.edu/~tao/247b.1.07w/notes8.pdf}.

\bibitem{20}
Akshay Venkatesh, Sparse equidistribution problems, period bounds, and subconvexity, \emph{Ann. of Math. {\bf172}} (2010), 989--1094, \url{http://annals.math.princeton.edu/wp-content/uploads/annals-v172-n2-p05-p.pdf}.

\end{thebibliography}

\end{document}